\documentclass{article} 
\usepackage{nips14submit_e,times}
\usepackage{hyperref}
\usepackage{url}

\usepackage{amsmath}
\usepackage{amssymb}
\usepackage{amsthm}
\usepackage{amscd}
\usepackage{amsfonts}
\usepackage{graphicx}%
\usepackage{fancyhdr}
\usepackage{geometry}
\usepackage{color}

\usepackage{times}
\usepackage{subfigure}

\usepackage[authoryear]{natbib}

\usepackage{algorithm}
\usepackage{algorithmic}

\newtheorem{prop}{Proposition}
\newtheorem{coro}{Corollary}

\newtheorem{theorem}{Theorem}
\newtheorem{example}{Example}
\newtheorem{lemma}{Lemma}
\newcommand{\n}[0]{\hspace*{.35em}}

\newcommand{\nn}[0]{\hspace*{.7em}}

\title{Statistical Models for Degree Distributions of Networks}

\author{
Kayvan Sadeghi \\
Department of Statistics\\
Carnegie Mellon University\\
Pittsburgh, PA 15213 \\
\texttt{kayvans@andrew.cmu.edu} \\
\And
Alessandro Rinaldo \\
Department of Statistics\\
Carnegie Mellon University \\
Pittsburgh, PA 15213 \\
\texttt{arinaldo@cmu.edu} \\
}

\nipsfinalcopy 

\begin{document}

\maketitle

\begin{abstract}
We define and study the statistical models in exponential family form whose
 sufficient statistics are  the degree distributions and the  bi-degree
 distributions of undirected labelled simple graphs. Graphs that are constrained
 by the joint degree distributions are called $dK$-graphs in the computer
 science literature and this paper attempts to provide the first statistically
 grounded analysis of
 this type of models. In addition to formalizing these models,  we provide
 some preliminary results for the parameter estimation and the asymptotic
 behaviour of the model for degree distribution, and discuss the parameter
 estimation for the model for bi-degree distribution.
\end{abstract}

\section{Introduction and background} \paragraph{Introduction.}
\emph{Exponential random graph models} (ERGMs) form a flexible and powerful family of statistical models for network data, used in variety of fields
and especially in the social scenes; see \cite{rob07} and other papers within the same special issue. These models are of exponential family form with proposed sufficient statistics that can range from the number of edges of the network \cite{erd59} to the number of $k$-stars, or other graphical features of networks; see for example \cite{hol81,fra86}.

While ERGMs can be specified using any network statistics, it has been long
known  that  node degrees
 and statistics thereof have great expressive power in representing and
 modeling networks, perhaps more than most network statistics.  See, e.g.,
 \cite{new03} and \cite{han07}. Some of the recent literature on ERGMs have
 focussed on the properties of the \emph{beta model}, for which the degree
 sequence of a network is the sufficient statistic and postulates independent
 edges; see \cite{bli10},  \cite{cha11}, and \cite{rin13}.
In this paper we consider instead the less studied class of ERGMs whose
sufficient statistics are derived from the {\it joint degree distribution} of the
nodes, for which the assumption of dyadic independence no longer holds.

This work is motivated in great part by the desire to describe the statistical
properties of a
class of network statistics known as  \emph{$d$K-graphs}, originally proposed by
\cite{mah06}. $d$K-graphs
were originally formulated as a means to capture increasingly refined properties of
networks in a hierarchical manner based on higher order interactions among node
degrees (see, e.g., \cite{dim09}).
Despite the significant appeal of $d$K graphs as summary statistics of networks
and their ability to capture stochastic dependencies among noes (see, in
particular,  \cite{shi11}),
to the best of our knowledge, the statistical properties of such statistics have not
been investigated. The purpose of this paper is to uninitiate such study and to
offer some preliminary but non-trivial results that highlight the
complexity and modeling power of these statistics. Like it is often the case
with ERGMs relying on network statistics that are not based on dyadic
independence, the theoretical analysis is particularly challenging.

\paragraph{Contributions.} This paper contains the following
contributions. We formally define ERGMs based on $d$K sequences of
different dimensions. For the $1$K and $2$K models, we derive conditions
for the existence of the MLE of the model parameter and, therefore, for their
estimability. We consider  the general case of $m\geq 1$ i.i.d observations,
which in particular includes the more interesting and common  case of $m=1$
observed network. We are concerned with the asymptotic behavior of the above model and
compare it specifically to the behavior of the dense Erd\"{o}s-R\'{e}nyi model.
We show that the $1K$ model is in fact radically different from  the
Erd\"{o}s-R\'{e}nyi model, an appealing feature that  many
ERGMs do not always possess, as demonstrated by
\cite{cha13}. Finally, we define the ERGM with bi-degree distribution and
provide ideas and conjectures for further work for such models.

\paragraph{Basic definitions and concepts.} We denote the set of simple,
undirected, labeled graphs with $n$
nodes by $\mathcal{G}_n$. If in an iid sample there is at least one observation
of each members of $\mathcal{G}_n$, a natural estimation of the probability of
observing $g\in\mathcal{G}_n$ is the ratio of the number of observations of $g$
to the total number of observations. However, in most practical cases, there are
considerably fewer observations than this, and in the most common case, there is only one observation available.

For the above purpose, we propose a family of ERGMs in order to extract as much information as possible on the (joint) degrees of nodes of the few available observations. More precisely, the sufficient statistics of a model in this family are scaled forms of the number of induced connected subgraphs with a specific number of nodes with the same degree sequence. We call such subgraphs \emph{configurations}. For example, for one-node configurations of the graph $g$ with $n$ nodes, that is, nodes of $g$, the sufficient statistics are the number of nodes with degrees $(0,\dots,n-1)$, denoted by $n^{(1)}(g)=(n_0(g),n_1(g),\dots,n_{n-1}(g))$; and  for two-node configurations of $g$, that is, edges of $g$, the sufficient statistics are the number of edges with degrees $((1,1),(1,2),\dots,(n-2,n-1),(n-1,n-1))$, denoted by $n^{(2)}(g)=(n_{11}(g),n_{12}(g),\dots,n_{n-1,n-1}(g))$, whose components are indexed lexicographically. The three-node configurations of $g$ are triangles (i.e., complete subgraphs with three nodes) and $2$-stars (i.e., subgraphs with three nodes, in which a node is connected to two non-adjacent nodes) of $g$, and so on. For example, in the graph $g$ in Figure 1, $n^{(1)}(g)=(n_0(g),n_1(g),n_2(g),n_3(g))=(0,1,2,1)$ and $n^{(2)}(g)=(n_{11}(g),n_{12}(g),n_{13}(g),n_{22}(g),n_{23}(g),n_{33}(g))=(0,0,1,1,2,0)$.
\begin{figure}[H]\label{fig:0}
\centering
\includegraphics[scale=0.07]{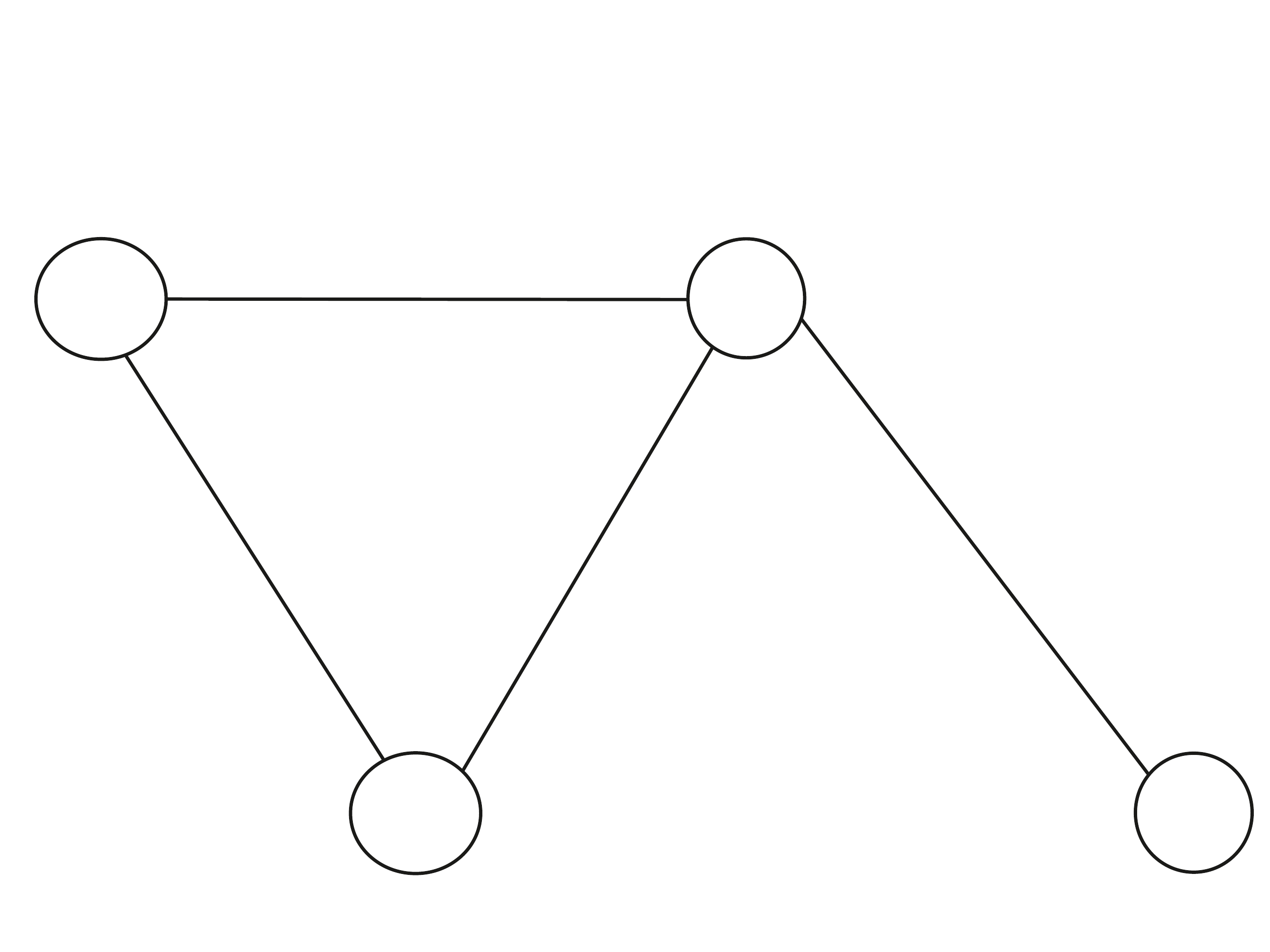}
\caption{{\footnotesize Graph $g$ with $n^{(1)}(g)=(0,1,2,1)$ and $n^{(2)}(g)=(0,0,1,1,2,0)$.}}
\end{figure}
In general, the vector $n^{(d)}(g)$ of length ${n\choose d}$ can be defined in the same fashion for connected subgraphs with $d$ nodes (i.e.\ \emph{configurations}), and it can be represented in a normalized form, scaled by the number of all $d$-node configurations in the graph. For example, since $\sum_{k=0}^{n-1}n_{k}(g)=n$, instead of $n^{(1)}(g)$, one can use $n^{(1)}(g)/n$, which is generally called the \emph{degree distribution} of $g$. In higher orders, $n^{(d)}(g)$ is called the \emph{joint degree distribution} of $g$, and in particular the normalized $n^{(2)}(g)$ is called the \emph{bi-degree distribution} of $g$.

Notice also that $n^{(1)}(g)$ is another way of expressing, and hence contains no more information than, the \emph{degree partition} of the unlabelled graph $g$, that is, $d(g)=(d_1(g),\dots,d_n(g))$, where each $d_i$ represents the degree of a node in $g$, usually ordered in  a way such that $d_1(g)\geq d_2(g)\geq\dots\geq d_n(g)$. The degree sequence can also be generalized for higher orders to give the \emph{joint degree partition} of a graph.

\paragraph{The $d$K-graph models.} The set of graphs constrained by the joint degree distribution is called the \emph{$d$K-graphs}, where $d$ is the same as above, standing for the number of nodes of the configuration. The joint distribution itself is sometimes called the $d$K-distribution. As a convention, the $0$K-distribution is the \emph{average degree}, that is, the ratio of the number of edges to the number of nodes in the graph. Thus $n^{(0)}(g)$ can be defined as $e(g)$, the number of edges of $g$.  The corresponding model is equivalent to the Erd\"{o}s-R\'{e}nyi model after reparameterizing  $p$ in the $0$K model with its odd $p/(1-p)$.


For the above reason, we refer to the proposed statistical models in exponential family from $\exp\{\sum_{k=0}^{n-1}s_k(g)\alpha_k-\psi(\alpha)\}$, with sufficient statistics $s_k(g)=n^{(d)}(g)$ as the \emph{$d$K models}. These models assign the same probability of occurrence to all graphs with the same joint degree distributions of corresponding order.

These models are also hierarchical in the sense that  from $n^{(d)}$, all $n^{(d')}$, $0\leq d'\leq d$, are uniquely determined, whereas the converse is not true. For example, the following observation provides a method for obtaining $n^{(0)}$ from $n^{(1)}$:
\begin{equation}\label{eq:00}
e(g)=\frac{1}{2}\sum_{k=0}^{n-1}kn_k(g);
\end{equation}
and the following observations provide a method for obtaining $n^{(1)}$ from $n^{(2)}$:
\begin{subequations}
\begin{equation}
n_k(g)=\frac{1}{k}(\sum_{k'=1}^{k}n_{k'k}(g)+\sum_{k'=k}^{n-1}n_{kk'}(g)),k\in\{1,\dots,n-1\}; \label{eq:01}\\
\end{equation}
\begin{equation}
n_0(g)=n-\sum_{k=1}^{n-1}n_k(g).\label{eq:02}
\end{equation}
\end{subequations}
Hence by moving from a higher order model to a lower one, one loses some information that can be inferred from the data. This can be seen as the fact that, by moving towards a lower order, there would be a larger sets of graphs with the same sufficient statistics. Lower order models, on the other hand, enable us to infer the structure of a larger set of possible graphs based on the more local characteristics of the observed graphs. Bespeaking, for  lower order models, more parameters can be estimated when there are no observations of some joint degree distributions at all.


However, in practice, there are usually very few observations, and in many cases only one observation available; hence it is normally more practical to work with the $1$K or $2$K models. This is the direction we take in this paper, and hence we only focus on the $1$K and $2$K models.



\paragraph{Structure of the paper.} In the next section, we study the $1$K model, discussing the calculation of the normalizing constant, the existence and calculation of maximum likelihood estimation, as well as the asymptotic behavior of the MLE. In Section 3 we apply a parallel theory in order to define the $2K$ model, and we present some ideas and conjectures regarding the maximum likelihood estimation for further work.

\section{The $1$K model}
\paragraph{The exponential family form.} As mentioned before, our goal is to model the probability of observing a network $g$ with $n$ nodes. Denote the probability of a node having degree $k$ by $p_k$. Considering the ordered vector of possible degrees $(0,\dots,n-1)$ in $g$, the probability of observing $g$ can be written as
\begin{equation}\label{eq:1}
P(g)=\varphi(p)\prod_{k=0}^{n-1}p_k^{n_k(g)},
\end{equation}
where $n_k(g)$ is the number of nodes with degree $k$ in $g$, $\varphi$ is a \emph{normalizing constant}, and $\sum_{k=0}^{n-1} p_k=1$.

To write (\ref{eq:1}) in a minimal form, we use the parametrization $\tilde{p}_k=p_k/p_{n-1}$, which by using the fact that $\tilde{p}_{n-1}=1$, implies $p_{n-1}=1/(1+\sum_{k=0}^{n-2}\tilde{p}_k)$. Therefore, (\ref{eq:1}) can be rewritten as
\begin{equation}\label{eq:2}
P(g)=\varphi(p)p_{n-1}^{\sum_{k=0}^{n-1}n_k(g)}\prod_{k=0}^{n-1}\tilde{p}_k^{n_k(g)}=\frac{\phi(\tilde{p})}{(1+\sum_{k=0}^{n-2}\tilde{p}_k)^n}\prod_{k=0}^{n-2}\tilde{p}_k^{n_k(g)}.
\end{equation}
This in turn can be parametrized in exponential family form as
\begin{equation}\label{eq:3}
P(g)=\exp \left\{ \sum_{k=0}^{n-2}n_k(g)\alpha_k-\psi(\alpha) \right\},
\end{equation}
where $\alpha_k=\log \tilde{p}_k$ and $\psi(\alpha)=n\log(1+\sum_{k=0}^{n-2}\exp(\alpha_k))-\log\phi(\tilde{p})$.

 We reduced the dimension of the sufficient statistics by arbitrarily removing the element $n_{n-1}$ , to obtain $n^{(1)}_-(g)$ as sufficient statistics. We see that this model assigns the same probability to graphs with the same degree distribution, which is the only information the model collects from the data.
\paragraph{Calculating the normalizing constant.} We know that $\sum_{g\in\mathcal{G}_n}P(g)=1$, where $\mathcal{G}_n$ is the set of all non-isomorphic graphs with $n$ nodes. We use this to calculate the normalizing constant $\psi(\alpha)=\psi_n(\alpha)$ for a fixed $n$.
In this case the normalizing constant can be calculated directly from the set of all graphs with $n$ nodes. Denote all these graphs by $\mathcal{G}_n=\{g_1,\dots,g_M\}$, where $M=2^{{n\choose 2}}$.
\begin{prop}\label{prop:1}
The normalizing constant for the $1$K model can be written as
\begin{equation}\label{eq:5}
\psi(\alpha)=\log(1+\dots+\prod_{k=0}^{n-2}\tilde{p}_k^{n_k(g_i)}+\dots+\tilde{p}_0^n)=\log(1+\dots+e^{\sum_{k=0}^{n-2}n_k(g_i)\alpha_k}+\dots+e^{n\alpha_0}),
\end{equation}
where there are $M$ terms corresponding to each $g_i\in\mathcal{G}_n$, and the first and the last terms correspond to the complete and the null graphs.
\end{prop}
\begin{proof}
Assume that $g_1$ is the complete graph, $g_M$ is the null graph, and $g_i$ is an arbitrary graph with $n$ nodes.  Equation (\ref{eq:2}) implies that $P(g_1)=\phi(\tilde{p})/(1+\sum_{k=0}^{n-2}\tilde{p}_k)^n$, $P(g_M)=\tilde{p}_0^n\cdot\phi(\tilde{p})/(1+\sum_{k=0}^{n-2}\tilde{p}_k)^n$, and $P(g_i)=\prod_{k=0}^{n-2}\tilde{p}_k^{n_k(g_i)}\cdot\phi(\tilde{p})/(1+\sum_{k=0}^{n-2}\tilde{p}_k)^n$. Now by using $\sum_{i=1}^MP(g_i)=1$ we obtain $\phi(\tilde{p})=(1+\sum_{k=0}^{n-2}\tilde{p}_k)^n/(1+\dots+\prod_{k=0}^{n-2}\tilde{p}_k^{n_k(g_i)}+\dots+\tilde{p}_0^n)$. This, together with $\alpha_k=\log \tilde{p}_k$ and $\psi(\alpha)=n\log(1+\sum_{k=0}^{n-2}\exp(\alpha_k))-\log\phi(\tilde{p})$, implies (\ref{eq:5}).
\end{proof}
Notice that in (\ref{eq:5}) there are repeated terms for different labeling of
isomorphic graphs as well as non-isomorphic graphs with the same degree
distribution. We illustrate the proposition with an example; see also \cite{mckay:85}
for asymptotic estimates of the number of graphs in $\mathcal{G}_n$ with a
prescribed degree distribution.
\begin{example}\label{ex:1}
There are $4$ non-isomorphic graphs with $3$ nodes, where $g_2$ and $g_3$ are repeated three times for different labeling:
\begin{figure}[H]\label{fig:1}
\begin{tabular}{cccc}
  \includegraphics[scale=0.065]{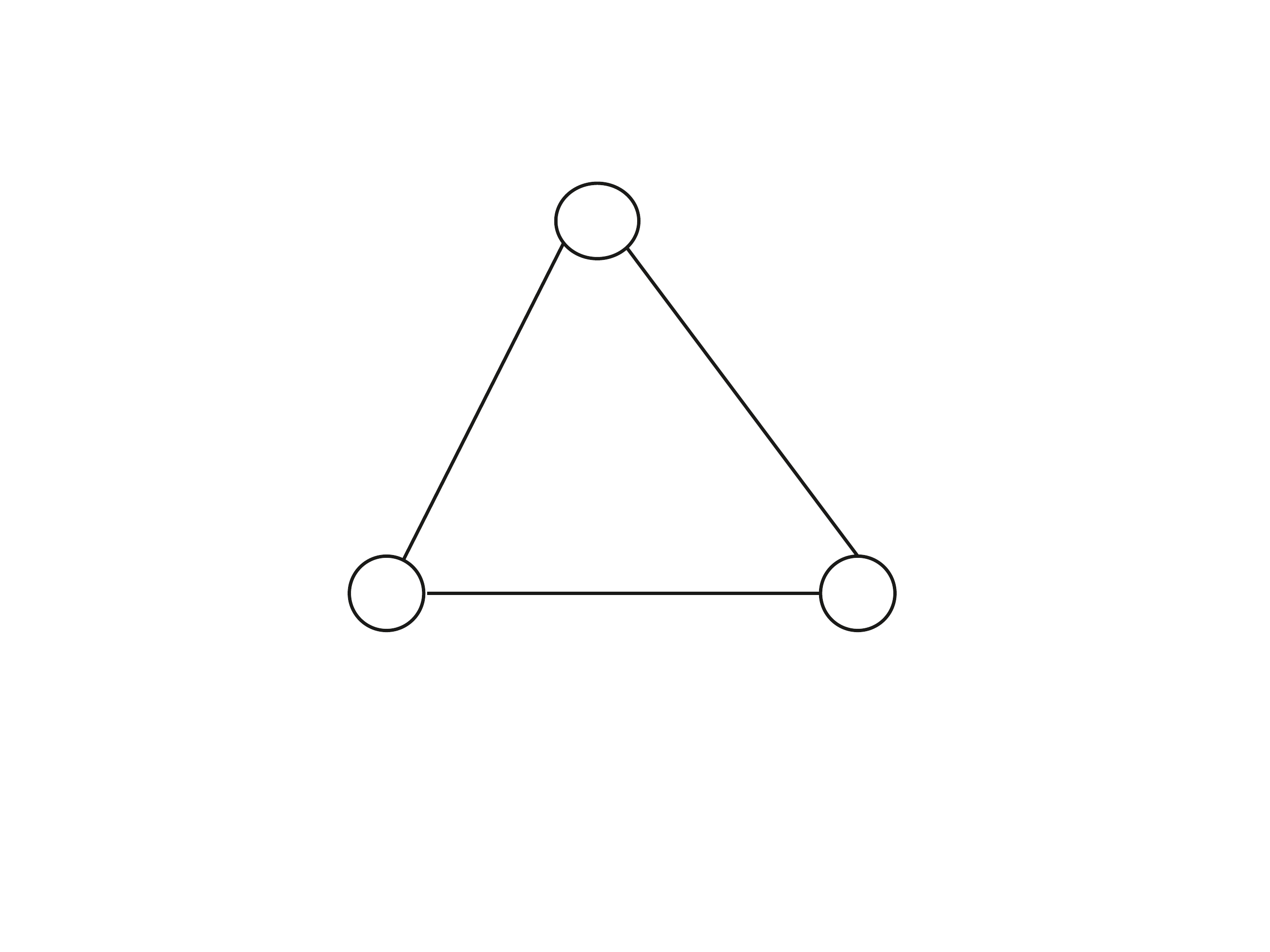}\nn\nn &\nn\nn \includegraphics[scale=0.065]{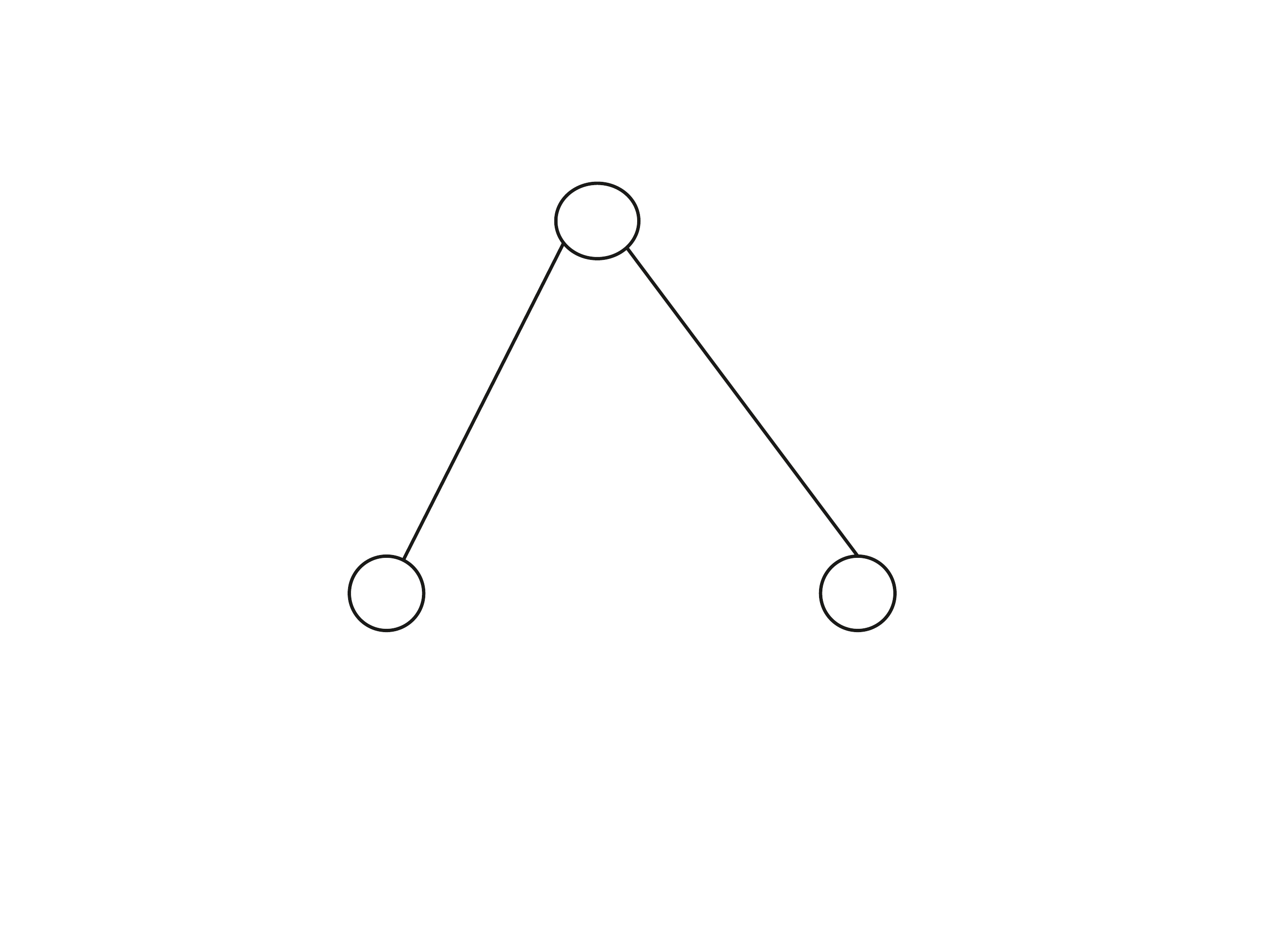}\nn\nn &\nn\nn \includegraphics[scale=0.065]{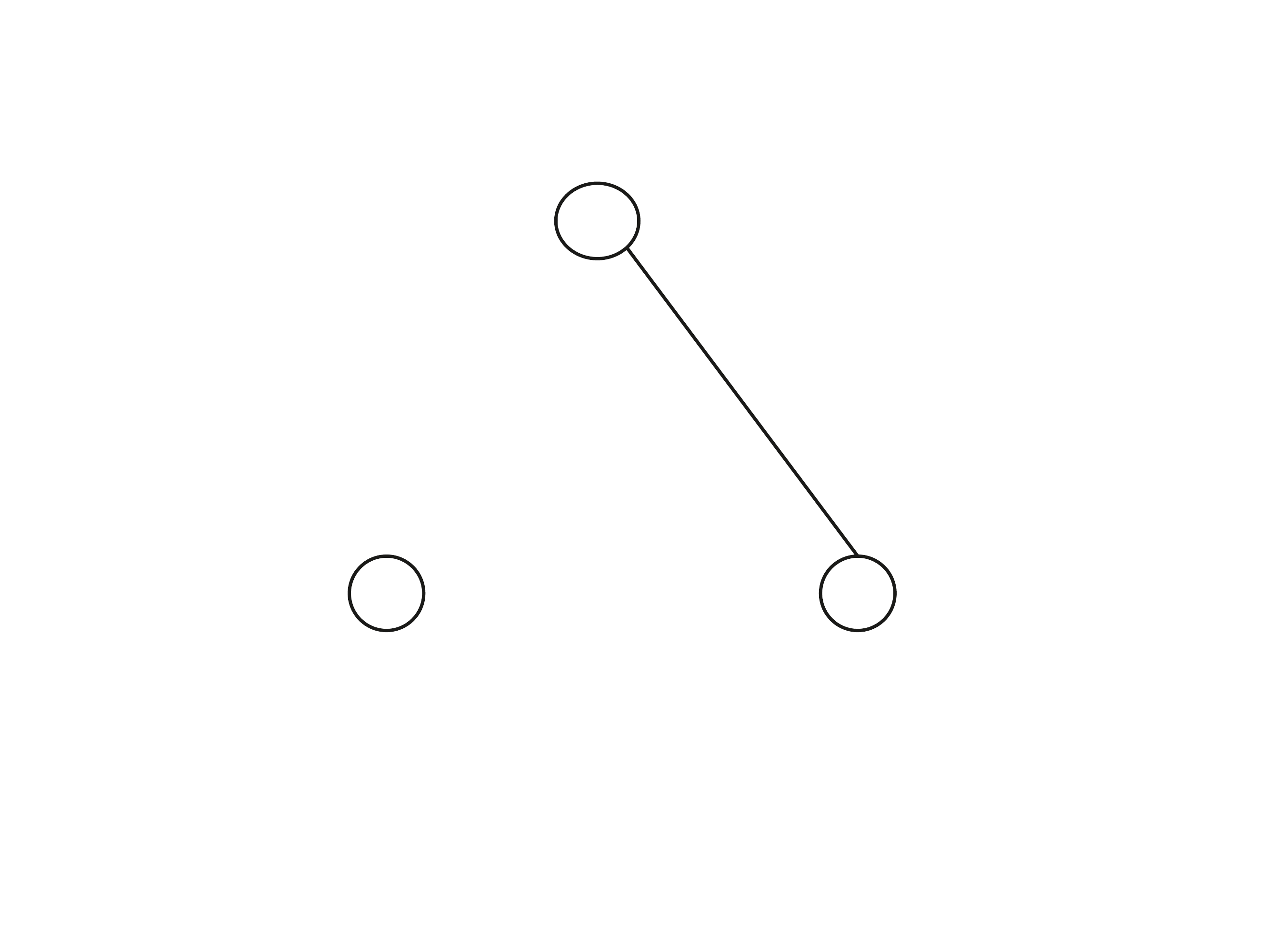} & \nn\nn \includegraphics[scale=0.065]{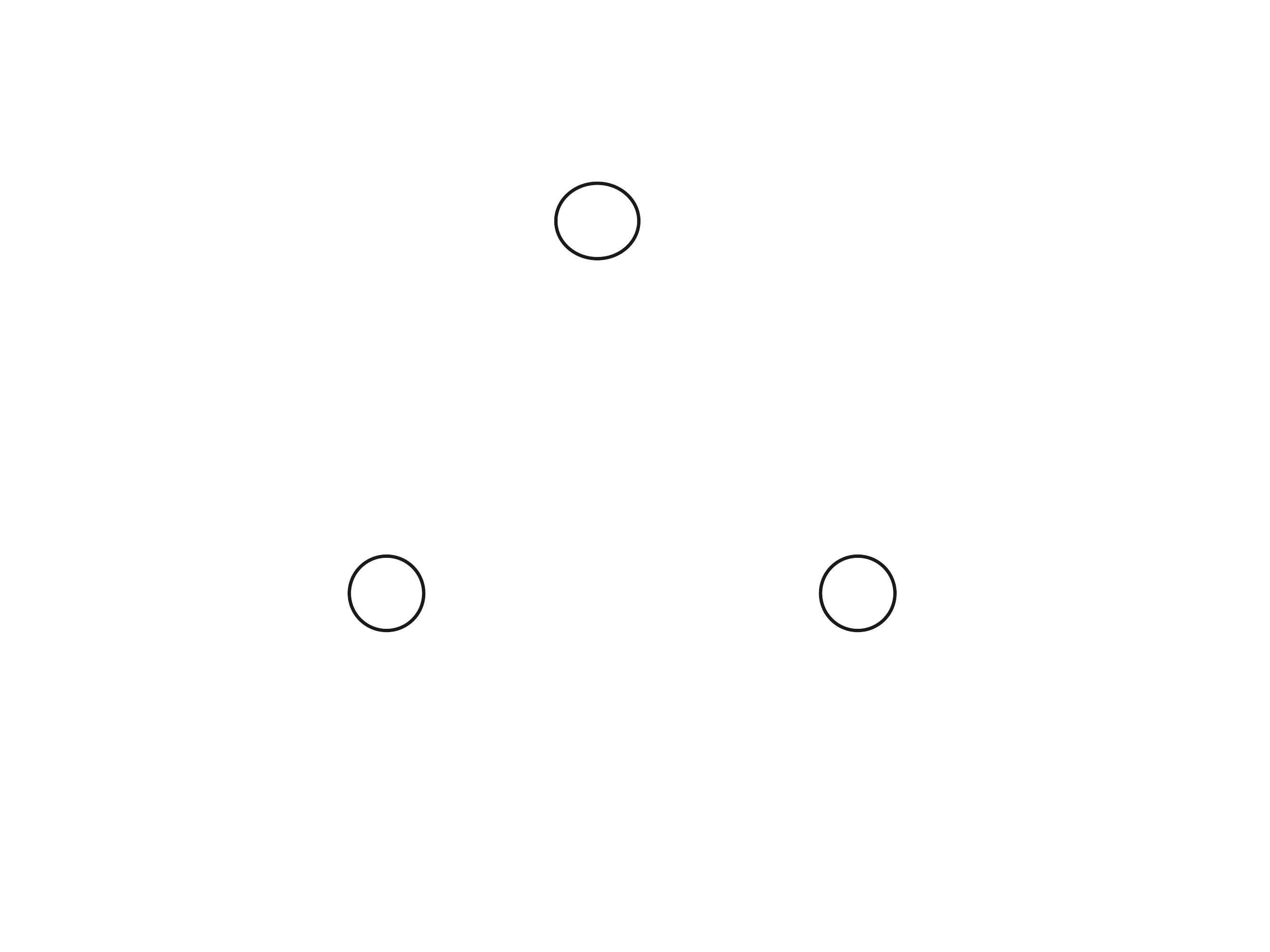}
  \\
   $g_1$\nn\nn &\nn\nn $g_2$ $(3)$\nn\nn &
\nn\nn  $g_3$ $(3)$\nn\nn &\nn\nn $g_4$
\end{tabular}
\caption{{\footnotesize All non-isomorphic graphs with $3$ nodes and the number of times they are repeated due to labeling.}}
\end{figure}
Now Proposition \ref{prop:1} implies
\begin{equation}\label{eq:4}
\psi(\alpha)=\log(1+3\tilde{p}_1^2+3\tilde{p}_1^2\tilde{p}_0+\tilde{p}_0^3)=\log(1+3e^{2\alpha_1}+3e^{2\alpha_1+\alpha_0}+e^{3\alpha_0}).
\end{equation}
This can be verified by considering (\ref{eq:2}), which implies that $P(g_1)=\phi(\tilde{p})/(1+\tilde{p}_0+\tilde{p}_1)^3$, $P(g_2)=\tilde{p}_1^2\cdot \phi(\tilde{p})/(1+\tilde{p}_0+\tilde{p}_1)^3$, $P(g_3)=\tilde{p}_1^2\cdot\tilde{p}_0\cdot\phi(\tilde{p})/(1+\tilde{p}_0+\tilde{p}_1)^3$, and $P(g_4)=\tilde{p}_0^3\cdot\phi(\tilde{p})/(1+\tilde{p}_0+\tilde{p}_1)^3$, and by following the calculations in the proof of Proposition \ref{prop:1}.
\end{example}


\paragraph{Degeneracy.} Heuristically, the $1K$ model should be regarded as relatively
unaffected by degeneracy issues, for at least two reasons. First, the change statistic \cite[see][]{Snijders2006New}
corresponding to adding an edge between two nodes of degrees $k$ and $k'$ is
$\alpha_{k+1} + \alpha_{k'+1} - (\alpha_k + \alpha_{k'})$, suggesting in general
a lack of
significant
correlation with the number of edges in the graph. Secondly, the growing (in $n$)
number of parameters and the fact that the degree distributions are
negatively correlated prevent the distribution from concentrating on very few
configurations as it pushes the probability mass to spread out.
\paragraph{Existence of the maximum likelihood estimator.} Suppose that $\{g_1,\dots,g_m\}$, $m\geq 1$ are $m$ iid observations of the networks with $n$ nodes. The \emph{log-likelihood} function can be written as
 \begin{equation}\label{eq:6}
l({\alpha|g_1,\dots,g_m})=(\sum_{k=0}^{n-2}\alpha_k\sum_{l=1}^mn_k(g_l))-m\psi(\alpha).
\end{equation}

We define the \emph{average observed sufficient statistics} for the $1$K model to be $\bar{n}^{(1)}=(\sum_{r=1}^mn^{(1)}(g_r))/m$ for $m$ iid observations $\{g_1,\dots,g_m\}$.  It is known that, for distributions in exponential families, the MLE exists if and only if the average observed sufficient statistics lie on the interior of the model polytope; see \cite{bar78} and \cite{bro86}.

We first explore the corresponding \emph{model polytope} $A_{n-1}=convhull(\{n^{(1)}_-(g),g\in\mathcal{G}_n\})$ to derive a necessary and sufficient condition for the existence of the maximum likelihood estimator (MLE) of $\alpha$, that is, $\{\hat{\alpha}:l(\hat{\alpha}|g_1,\dots,g_m)=\max_{\alpha\in\mathbb{R}}l(\alpha|g_1,\dots,g_m)\}$. We then exploit the well-studied theory of exponential family to calculate the MLE when existing.

In order to study $A_{n-1}$, we need the two following notations and lemmas: Denote the $k$-regular graphs with $n$ nodes by $R_k$. In addition, by $R_{kl}$ denote the graphs with $n$ nodes such that $n_{k}(R_{kl})=n-1$ and $n_{l}(R_{kl})=1$ .
\begin{lemma}\label{lem:1}
For $0<k<n$, there exists an $R_k$ if and only if $kn$ is even.
\end{lemma}
\begin{lemma}\label{lem:2}
Suppose that $n$ and $k$ are odd numbers and $0\leq k,l<n$. There exists an $R_{kl}$ if and only if $l$ is even.
\end{lemma}
\begin{proof}
Suppose that $l$ is even. We know that there is a $k$-regular graph $H$ with $n-1$ nodes. For this graph, consider a matching $\{(i_1,j_1),\dots,(i_{(n-1)/2},j_{(n-1)/2}\}$. Now $H$ and an isolated node $h$ provides $R_{k0}$. By removing the edge between $i_1$ and $j_1$ and connecting $h$ to both $i_1$ and $j_1$, we obtain $R_{k2}$. By this method, we can inductively generate all $R_{kl}$.

If $l$ is odd, the sum of degrees of nodes would be an odd number, which is impossible.
\end{proof}
Let $B_n=convhull(\{n^{(1)}(g),g\in\mathcal{G}_n\})$. In addition, for $0\leq k<n$, let $e_k=(0,\dots,0,n,0,\dots,0)$, where $n$ is the $(k+1)$st element of the vector of length $n$, and for $0\leq k,l<n$ and $k\neq l$, let $e_{kl}=(0,\dots,0,1,0,\dots,0,n-1,0,\dots,0)$, where $n-1$ is the $(k+1)$st element of the vector and $1$ is the $(l+1)$st element of the vector of length $n$. We observer that $n^{(1)}(R_k)=e_k$ and $n^{(1)}(R_{kl})=e_{kl}$. The following lemma characterizes $B_n$; see also
\begin{lemma}\label{lem:3}
The model polytope $B_n$ is
\begin{description}
  \item[(if $n$ is even)] $n\triangle^{n-1}$, that is, the $(n-1)$-simplex scaled by $n$; and
  \item [(if $n$ is odd)] the convex hall of the set of extreme points $\{e_{l}:0\leq k<n, \text{l even}\}\cup\{e_{k,l}: 0\leq k<n, 0\leq l<n,\text{k odd, l even}\}$.
\end{description}
\end{lemma}
\begin{proof}
$\sum_{k=0}^{n-1}n_{k}(g)=n$ implies that $A_{n-1}\subseteq n\triangle^{n-1}$. We now need to deal with the two cases separately:

1) \emph{$n$ even:} By lemma \ref{lem:1}, we know that  for each $0\leq k<n$, there exists an $R_k$. The corresponding vectors $e_k$  are all the vertices of $A_{n-1}$; therefore, $A_{n-1}=n\triangle^{n-1}$.

2) \emph{$n$ odd:} For $0\leq l<n$ and $l$ even, by lemma \ref{lem:1}, there exist $R_l$, and the corresponding $e_l$ are the extreme points. For $0\leq k,l<n$, odd $k$ and even $l$, by Lemma \ref{lem:2}, there exist $R_{kl}$, and it is easy to see that the corresponding $e_{kl}$ cannot be generated by the convex combination of other $n^{(1)}(g)$. Therefore, these are the extreme points too. Now only the vectors with integer entries fewer then $n$ that lie on $n\triangle^{n-1}$ but not on the convex combination of extreme points contain an element $n$ as an entry, which, again by Lemma \ref{lem:1}, it is not possible.
\end{proof}

Let $\mathbf{0}_{i}$ be the vector of size $i$ consisting only of zero elements. We can now characterize $A_{n-1}$:
\begin{prop}\label{prop:2}
The model polytope $A_{n-1}$ is
\begin{description}
  \item[(if $n$ is even)] the convex hall of the set of extreme points $\{e_{k}:0\leq k<n-1\}\cup\{\mathbf{0}_{n-2}\}$; and
  \item [(if $n$ is odd)] the convex hall of the set of extreme points $\{e_{l}:0\leq l<n-1, \text{l even}\}\cup\{e_{k,l}: 0\leq k<n-1, 0\leq l<n-1,\text{k odd, l even}\}\cup\{\mathbf{0}_{n-2}\}$.
\end{description}
\end{prop}
\begin{proof}
By projecting the polytope $B_n$, given in Lemma \ref{lem:3}, onto the $(n-1)$-dimensional Euclidean space with coordinates $(n_0,\dots n_{n-2})$, we obtain the result.
\end{proof}
In Figure 3, the polytopes $B_3$ and $A_3$, for graphs with three nodes, are depicted.

\begin{figure}[H]\label{fig:4}
\begin{tabular}{cc}
  \includegraphics[scale=0.15]{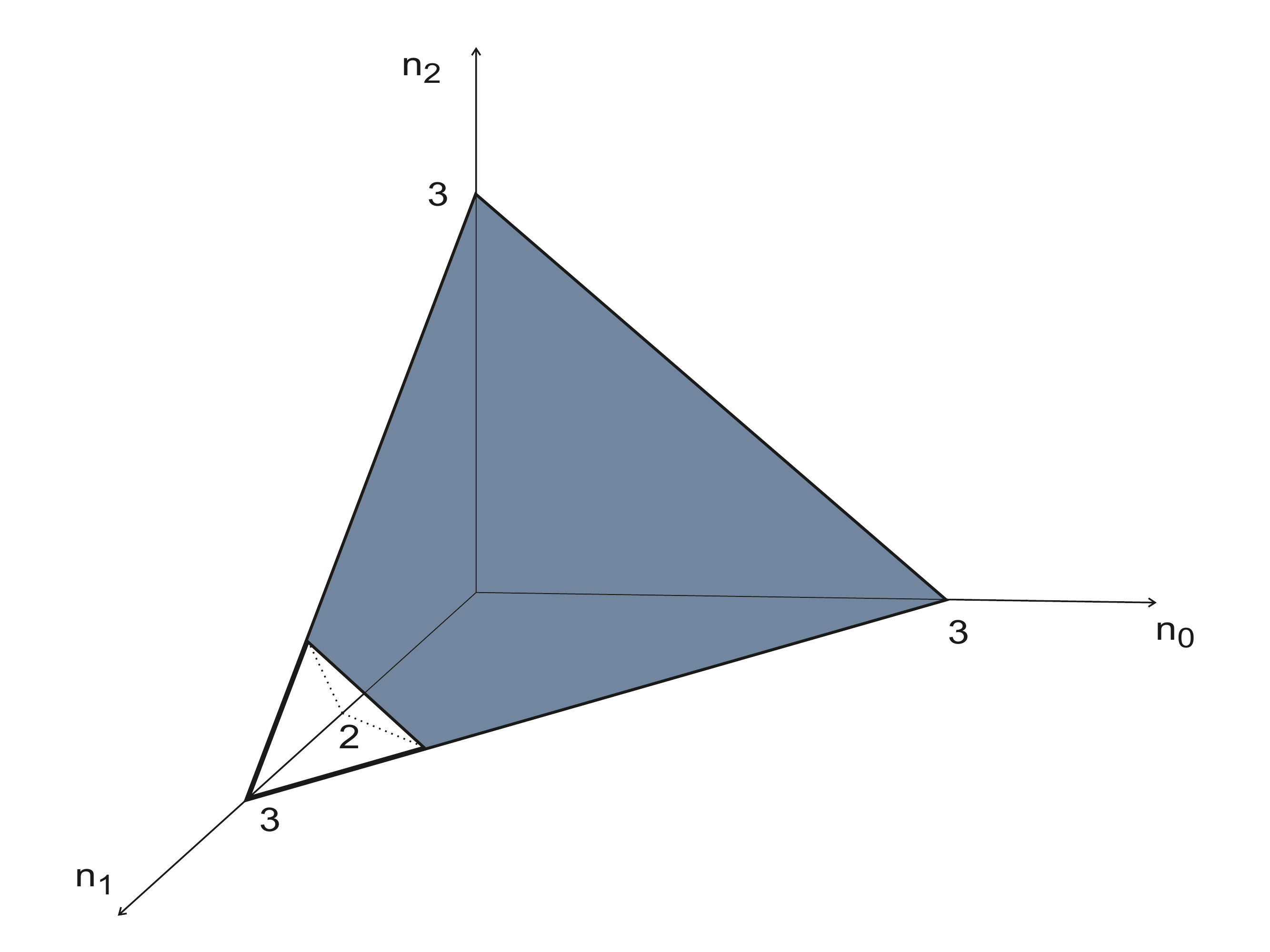}\nn\nn\nn & \nn\nn\nn\includegraphics[scale=0.17]{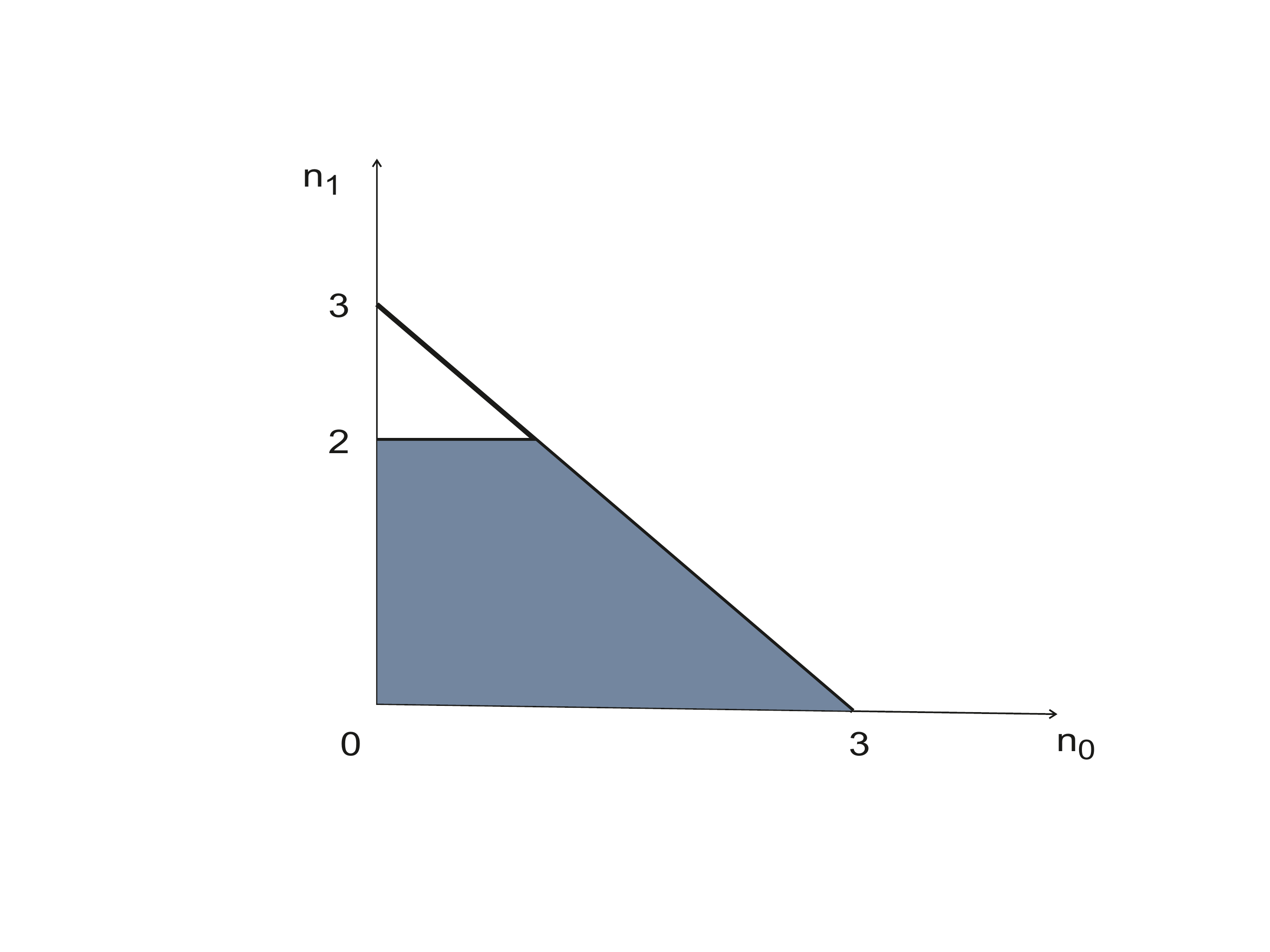}
  \\
   $(a)$\nn\nn\nn &\nn\nn\nn $(b)$
\end{tabular}
\caption{{\footnotesize (a) The polytope $B_3$. (b) The polytope $A_3$. }}\label{fig:4n}
\end{figure}

Therefore, we have the following theorem:
\begin{theorem}\label{thm:1}
For the $1$K model and $m$ iid observations $\{g_1,\dots,g_m\}$, the MLE exists if and only if
\begin{description}
  \item[(if $n$ is even)] $\bar{n}^{(1)}_k\neq 0$ for all $k$, $0\leq k\leq n-2$, and $\sum_{k=0}^{n-2}\bar{n}^{(1)}_k<n$; and
  \item [(if $n$ is odd)] (i) $\bar{n}^{(1)}_r\neq 0$ for all $r$, $0\leq r\leq n-2$, and $\sum_{r=0}^{n-2}\bar{n}^{(1)}_r<n$; and (ii) for every odd $k$, $\bar{n}^{(1)}_k<n-1$.
\end{description}
\end{theorem}
\begin{proof}
The MLE exists if and only if  $\bar{n}^{(1)}\in int(A_{n-1})$.

In the case that $n$ is even, a point in  $int(A_{n-1})$ is written as $a_0e_0+a_1e_1+\dots+a_{n-2}e_{n-2}+c\mathbf{0}_{n-2}$, where $a_k>0$ for $0\leq k<n-1$, $c>0$, and $\sum_{k=0}^{n-2}a_k+c=1$. Therefore, $\bar{n}^{(1)}\in int(A_{n-1})$ if and only if $a_k=\bar{n}^{(1)}_k$  and $c=n-\sum_{k=0}^{n-2}\bar{n}^{(1)}_k$, which implies the result.

In the case that $n$ is odd, a point in  $int(A_{n-1})$ is written as $\sum_{\substack{l=0 \\ l \text{\n even}}}^{n-3}a_le_l+\sum_{\substack{k=1 \\ k \text{\n odd}}}^{n-2}\sum_{\substack{l=0 \\ l \text{\n even}}}^{n-3}b_{k,l}e_{k,l}+c\mathbf{0}_{n-2}$, where $a_l>0$ for even $l$, $0\leq l<n-1$, $b_{k,l}>0$ for odd $k$ and even $l$, $0\leq k,l<n-1$, $c>0$, and $\sum_{\substack{l=0 \\ l \text{\n even}}}^{n-3}a_l+\sum_{\substack{l=0 \\ l \text{\n even}}}^{n-3}\sum_{\substack{k=1 \\ k \text{\n odd}}}^{n-2}b_{k,l}+c=1$.

If $\bar{n}^{(1)}\in int(A_{n-1})$, for every odd $k$, $\sum_{l=0}^{n-3}b_{k,l}=n_k/(n-1)$, which implies $n_k>0$. In addition, we know that $\sum_{l=0}^{n-3}b_{k,l}<1$, which implies that $n_k<n-1$. For every even $l$, $a_l=(1/n)(n_l-\sum_{k=1}^{n-2}b_{k,l})$, which implies $n_l>0$. Now by using $\sum_{l=0}^{n-3}a_l+\sum_{l=0}^{n-3}\sum_{k=1}^{n-2}b_{k,l}<1$, we obtain $\sum_{r=0}^{n-2}n_r<n$.

Conversely, if conditions (i) and (ii) hold, we let $b_{k,l}=(2 n_k)/(n-1)^2$ and $a_l= (1/n)(n_l-(2/(n-1)^2\sum_{k=1}^{n-2}n_k)$. We conclude that $a_l>0$, $b_{k,l}>0$, and $\sum_{l=0}^{n-3}a_l+\sum_{l=0}^{n-3}\sum_{k=1}^{n-2}b_{k,l}<1$, which imply the result.
\end{proof}
\begin{coro}
For a single observation in the $1$K model, the MLE does not exist.
\end{coro}
\begin{proof}
For an observed $g$, if $n_{n-1}(g)\neq 0$, then $n_0(g)=0$ and vice vera. $n_{n-1}(g)=0$ implies $\sum_{k=0}^{n-2}\bar{n}_k=n$. Therefore, in either case the MLE does not exist.
\end{proof}
\begin{example}
We proceed with Example \ref{ex:1} for networks with $3$ nodes, illustrated in Figure 2. Theorem \ref{thm:1} implies that in order for the MLE to exist (i) $\bar{n}_0\neq 0$, $\bar{n}_1 \neq 0$, and $\bar{n}_1 + \bar{n}_0 < 3$; (ii) one should observe either $g_1$ or $g_4$ since both $g_2$ and $g_3$ are of form $R_{kl}$. If only $(i)$ or $(ii)$ holds, then only one parameter is estimable.
\end{example}
\paragraph{Estimable parameters.}
By partially differentiating the log-likelihood function in (\ref{eq:6}), we conclude that the MLE $\hat{\alpha}$ should satisfy the following system of equations:
  \begin{equation}\label{eq:7}
\frac{\partial \psi(\hat{\alpha})}{\partial \hat{\alpha}_k}=\bar{n}_k,\nn\nn\text{for}\nn k\in\{0.\dots,n-2\}.
\end{equation}
Therefore, from (\ref{eq:4}), the MLE for $1$K model with $m$ observations must satisfy the following system of equations:
\begin{equation}\label{eq:9}
\frac{\sum_{i=1}^Mn_k(g_i)e^{\sum_{l=0}^{n-2}n_l(g_i)\hat{\alpha}_l}}{\sum_{i=1}^Me^{\sum_{l=0}^{n-2}n_l(g_i)\hat{\alpha}_l}}=\bar{n}_k,\nn\nn k\in\{0,\dots,n-2\}.
\end{equation}
Notice that if for a fixed $k$, $\bar{n}_k=0$ then $\sum_{i=1}^Mn_k(g_i)e^{\sum_{l=0}^{n-2}n_l(g_i)\hat{\alpha}_l}=0$. We observe that $\hat{\alpha}_k$ is the only parameter that appears in all terms and does not appear in all terms of the denominator of the left hand side of (\ref{eq:9}). Therefore, $\hat{\alpha}_k=-\infty$. This corresponds to $\hat{\tilde{p}}_k=0$, or equivalently $\hat{p}_k=0$ since it is that assumed $p_{n-1}\neq 0$.

Thus, based on the MLE, the model implies that the probability of a node having a specific degree is zero when no node with that degree has been observed. Hence, it is plausible to remove such parameters from the model and simply focus on the submodel of exponential family form. The corresponding model polytope would then be simply the same, embedded in the Euclidian space with remaining coordinates. By this method we can still estimate a subset of parameters whose corresponding sufficient statistic is nonzero.

Notice that even with a single observation, there are sometimes a considerable number of parameters that can be estimated. The following proposition deals with the extreme case of such graphs.
\begin{prop}\label{prop:4}
There exists a graph $T_n$ of every size $n$ such that $n_0(T_n)=0$ and $n_k(T_n)\neq 0$ for $k\in\{1,\dots,n-1\}$.
\end{prop}
\begin{proof}
We prove the result by induction on the number of nodes. For the base, where $n=2$, the result holds for $K_2$. Suppose that there exists such a graph $T_n$. We prove it for $n+1$: Since $n_0(T_n)=0$ there is a $k'\in\{1,\dots,n-1\}$ such that $n_{k'}=2$ and $n_k=1$ if $k\neq k'$. We construct the graph $T_{n+1}$ as follows. We add a node to $T_n$ and start connecting it to the nodes starting from the node with the node with degree $n-1$ and stop when we reach $k'$. Now regardless of what degree the added node has, there are nodes of all degrees except zero in
graph.
\end{proof}
Notice, however, that the large sum in (\ref{eq:9})  entails the common problem with ERG modeling that the computing of the MLE for this
model ultimately requires MCMC methods, just like with most other ERGMs.

\paragraph{Asymptotics.} We have shown that we can estimate a subset of parameters whose corresponding sufficient statistic is nonzero. Here we calculate the probability of a sufficient statistic to be nonzero, and discuss its behaviour asymptotically. In addition, we provide the asymptotical expected value of the number of non-zero sufficient statistics.

Under the $1$K-model, for each $k\leq n-1$, we have the following:
\begin{equation}\label{eq:12}
    \mathbb{P}(n_k(G)>0) = \frac{\sum_{g\in\mathcal{G}_n:n_k(g)\neq
0}\prod_{k'=0}^{n-1}p_{k'}^{n_{k'}(g)}}{\sum_{g\in\mathcal{G}_n}\prod_{k'=0}^{n-1}p_{k'}^{n_{k'}(g)}}.
\end{equation}
%
%

We will now consider the very special case  where all $p_{k'}$ are equal, which implies that they are all equal to $1/n$. In this case we see that the model is the same as the Erd\"{o}s-R\'{e}nyi model with $p=1/2$.

\begin{prop}\label{prop:5}
    Suppose that $p_i=p$ for all $i$ and let $k = k(n)=(n-1)/2+a_n$, where $\{ a_n \}$ is a
    sequence such that $0 \leq a_n \leq (n-1)/2$ for all $n$. It then holds that
\begin{enumerate}
    \item  $\displaystyle\lim_{n \to \infty}\mathbb{P}(n_k(g)>0)=1$ if $a_n=O(\sqrt{n})$;
    \item $\displaystyle\lim_{n \to \infty}\mathbb{P}(n_k(g)>0)=0$, if $a_n =
      \Omega(\sqrt{n \log n})$.
\end{enumerate}
\end{prop}

\begin{proof}
By using (\ref{eq:12}), we have that
$P(n_k(g)>0)=\frac{\sum_{g\in\mathcal{G}_n:n_k(g)\neq 0}\prod_{k'=0}^{n-1}p^n}{\sum_{g\in\mathcal{G}_n}\prod_{k'=0}^{n-1}p^n}=\frac{|\{g\in\mathcal{G}_n:n_k(g)\neq 0\}|}{|\mathcal{G}_n|}$,
which is the same as the same probability in the Erd\"{o}s-R\'{e}nyi model with
$p=1/2$. Let $\lambda_k(n)=n{n-1 \choose k}(1/2)^{n-1}$. From the theorem 3.1 in
\cite{bol01}, we know that if $\displaystyle\lim_{n \to
\infty}\lambda_k(n)=0$, then $\displaystyle\lim_{n \to \infty}P(n_k(g)>0)=0$,
and if  $\displaystyle\lim_{n \to \infty}\lambda_k(n)=\infty$, then
$\displaystyle\lim_{n \to \infty}P(n_k(g)>0)=1$. We show that the latter holds
for $ a_n=O(\sqrt{n})$, and the former if $a_n = \Omega( \sqrt{n \log n})$. By using Stirling's approximation,
$\lambda_k(n) \sim \frac{n (n-1)^{n-\frac{1}{2}}}{\left( n-1-k \right)^{n-k-\frac{1}{2}}2^{n-\frac{1}{2}}k^{k+\frac{1}{2}}\sqrt{\pi}} = \left( \frac{n-1}{2} \right)^{n-\frac{1}{2}}\frac{n}{(\frac{n-1}{2}-a_n(k))^{\frac{n}{2}-a_n(k)}(\frac{n-1}{2}+a_n(k))^{\frac{n}{2}+a_n(k)} \sqrt{\pi}} \nonumber \asymp  \frac{n ^{n+1/2}}{2^n }\frac{1}{\left(\frac{n}{2}\right)^n \left( 1 -\frac{2 a_n(k)}{n} \right)^{\frac{n}{2}-a_n(k)} \left( 1 +\frac{2 a_n(k)}{n} \right)^{\frac{n}{2}+a_n(k) }}  =  \frac{\sqrt{n}}{ \left( 1 -\frac{2 a_n(k)}{n} \right)^{\frac{n}{2}-a_n(k)} \left( 1 +\frac{2 a_n(k)}{n} \right)^{\frac{n}{2}+a_n(k) }} =  \frac{\sqrt{n} } { \left( 1 - \frac{4 a_n^2}{n} \frac{1}{n} \right)^{\frac{n}{2}} \left( 1 - 2 \frac{a_n^2}{n} \frac{1}{a_n} \right)^{-a_n}  \left( 1 + 2 \frac{a_n^2}{n} \frac{1}{a_n} \right)^{a_n}}$.

If $a_n = O(\sqrt{n})$ the previous display is $\asymp \sqrt{n}$ and hence tends
to infinity. If $a_n = \Omega( \sqrt{n \log n})$ the term will vanish instead. The claim is proved.
\end{proof}

We now consider the number $N$ of non-zero entries in the degree distribution
and show that it is of order $O \left( \sqrt{n \log n}\right)$, which
corresponds to the number of estimable parameters in the model.

\begin{prop}\label{prop:6}
    Suppose that $p_i=p$ for all $i$ and let $N$ denote the number of entries in
    the degree distribution that are non-zero. Then,
    \[
	\mathbb{E}[N] = O \left( \sqrt{n \log n} \right).
    \]
\end{prop}
\begin{proof}
 Let $k_n = \sqrt{c n \log n}$ for some $c>0$ which will be set below. Assume that $(n-1)p > k_n$ and let $\mathcal{A}$ denote the event that there exists a node with degree less than $(n-1)p - k_n$ or more than $(n-1)p + k_n$. Let $D_i$ denote the degree of node $i$ and notice that $D_i \sim \mathrm{Bin}(n-1,p)$, for all $i$. Then    $\mathcal{A} =  \bigcup_i \{ | D_i - \mathbb{E}[D_i] | > k_n\}$.

By the union bound and Hoeffding inequality we get that, for any $r>0$, $\mathbb{P}(\mathcal{A}) \leq n^2 \exp\{ -2 c \log n\} = \frac{2}{n^r}$,
provided that $c = r/2$. 
It is now easy to see that the expected number of non-zero entries in the degree
distribution is of order $O(\sqrt{n \log n})$. In details,
$\mathbb{E}[N] = \mathbb{E}[N|\mathcal{A}]\mathbb{P}(\mathcal{A}) +  \mathbb{E}[N|\mathcal{A}^c]\mathbb{P}(\mathcal{A}^c) \leq \frac{n}{n^r} + \left(1 - \frac{2}{n^r} \right) \sqrt{n \log n} = O(\sqrt{n \log n}).$
\end{proof}

We have seen that the case $p_i = p$ for all $i$ and some $p \in (0,1)$
corresponds to the Erd\"{o}s-R\'{e}nyi model with $p=1/2$. Using \eqref{eq:3},
this is equivalent to having $\alpha_k = 0$, for all $k=0,1,\ldots, n-2$. We now
consider the slightly more general case in which the vector $\alpha =
(\alpha_0,\alpha_1,\ldots, \alpha_{n-2})$ belongs to a subset $A$ of
$\mathbb{R}^{n-1}$ such that $\log(c_n) \leq \alpha_i \leq \log(C_n) $, where
$c_n \in (0,1]$ and $C_n \in [1,\infty)$, for all $i$.  Then it is easy to see
    that, for any $\alpha \in A$,
    \begin{equation}\label{eq:Cncn}
 \mathbb{P}_{\alpha}(n_k(G) > 0) \leq \left(
 \frac{C_n}{c_n}\right)^n\mathbb{P}_0(n_k(G)>0) \leq \left(
 \frac{C_n}{c_n}\right)^n \lambda_k(n),
 \end{equation}
 where $\lambda_k(n)$ is given in the proof or proposition \ref{prop:5} and
 $\mathbb{P}_\alpha$ denotes a probability of the random graph $G$ when sampled
 from the 1K model with parameter $\alpha$.
Next, for a sequence $\{ a_n \}_{n=1,2,\ldots}$ such that  $0 \leq a_n \leq \frac{n}{2}$, let   $h_n = \left( 1 -\frac{2 a_n}{n} \right)^{\frac{1}{2}-\frac{a_n}{n}} \left( 1 +\frac{2 a_n}{n} \right)^{\frac{1}{2}+\frac{a_n}{n} }$.
Then, by proposition \ref{prop:5}, \eqref{eq:Cncn} yields that, for $k =
    k(n) = (n-1)/2 + a_n$,  $\mathbb{P}_{\alpha}(n_k(G) > 0)$ is bounded by a term that is asymptotically of order $\left( \frac{C_n}{c_n}\right)^n \frac{\sqrt{n}}{(h_n)^n}$.

Now let $b = \mathrm{liminf}_n h_n$. If $a_n = O(\sqrt{n})$, then $(h_n)^n =
    O(1)$ from the proof of proposition \ref{prop:5}, which yields a trivial bound. Thus assume that $ b \in (1,2]$ since $h_n \in [1,2]$ for all $n$. Then, provided that	$\frac{C_n}{c_n} < b$, the probability $ \mathbb{P}_{\alpha}(n_k(G) > 0)$ vanishes. (Notice that this
implies that $1 < C_n / c_n < 2$.)

\paragraph{Relation with the Erd\"{o}s-R\'{e}nyi model.}
Recently and somewhat surprisingly, \cite{cha13}  have shown that some
specifications of the ERGM family lead to models which asymptotically behave
like Erd\"{o}s-R\'{e}nyi (ER) models for appropriate choices of $p$.

Here we provide some results illustrating the relationships between the
ER model and the 1K model.  Below we parametrize the 1K model
using the natural parameter vector $\alpha = (\alpha_1,\ldots, \alpha_n) \in
\mathbb{R}^{n-1}$. Notice that
    we have expunged $n_0$ from the vector of sufficient statistics, a choice
    that entails no loss of generality.

First off, we make the easy observation that  the ER model with probability $p$ can be represented by setting $\alpha_i = i \theta, \quad i=1,\ldots,n-1$,
    where $\theta = \log \left( \frac{p}{1-p} \right)$.
Next, we show that the 1K model is dramatically different from the
Erd\"{o}s-R\'{e}nyi model in the sense of being almost singular to it whenever the natural parameters are uniformly bounded
in absolute value.

    Below, we will denote with $P_p$ the probability distribution of the
    Erd\"{o}s-R\'{e}nyi model with parameter $p \in (0,1)$ and with $P_{\alpha}$
    the
    probability distribution  of the 1K model with natural parameter vector
    $\alpha = (\alpha_1,\ldots,\alpha_{n-1}) \in \mathbb{R}^{n-1}$.

    \begin{prop}
Let $p \in (0,1)$ such that $\min \{ |p -1/2|, 1 - p, p \} > \epsilon$, for a
fixed, arbitrarily small $\epsilon > 0$. Then, there exists a sequence of
subsets $\mathcal{G}_n(p)$ of $\mathcal{G}_n$ such that, for any sequence $\{
\alpha_n \}$ such that $\alpha_n \in \mathbb{R}^{n-1}$ for all $n$ and  $\| \alpha_n \|_\infty = o(n)$,
\[
    \lim_n P_p(\mathcal{G}_n(p)) = 1 \quad \text{and} \quad
    \lim_n P_{\alpha_n}(\mathcal{G}_n(p)) = 0.
    \]
\end{prop}

\begin{proof}
Let $    \mathcal{G}_n(p)$ be the set of graphs such that $(n-1)p - C \sqrt{n \log n} \leq d_i(g) \leq (n-1) p +
    C\sqrt{ n \log n}$,  for all $i=1,\ldots,n$,
    where  $d_i(g)$ denote the degree
    of the $i$th node of graph $g$ and $C$ is a positive constant to be
    specified.

    By Heoffding's inequality, for any $c > 0$, there exists a $C = C(c)$ such that
    $P_p(\mathcal{G}_n(p)) \geq 1 - \frac{1}{n^c}$, for $n$ large
    enough that $C \sqrt{\frac{\log n}{n}} < \epsilon$.
     Next, the probability    $P_{\alpha}(\mathcal{G}_p) = \frac{\sum_{g \in \mathcal{G}_n(p)} e^{ \sum_{i=1}^{n-1}
\alpha_i n_i(g)}}{\sum_{g \in \mathcal{G}_n} e^{ \sum_{i=1}^{n-1}
\alpha_i n_i(g)}}$ can be bounded from above by $\frac{e ^{M n } |
    \mathcal{G}_n(p)|}{e^{-M n } |\mathcal{G}_n|} = e^{2M n}
    \mathbb{P}_{1/2}(\mathcal{G}_n(p))$, a calculation that follows from simple algebra and the fact that, for each $g \in \mathcal{G}_n$,  $\sum_{i=1}^{n-1} n_i(g)
    = n - n_0(g) \leq n$.

Then,  since $p$  is $\epsilon$-away from $1/2$, Theorem 2.3 in \cite{cha13b} (see also
    \cite{cha13}) yields that, for all $n$ large enough, $\mathbb{P}_{1/2}(\mathcal{G}_n(p)) \leq e^{- n^2 c}$, for an appropriate constant $c$, which depends on $p$.
Thus, we have shown that, for all $n$ large enough and for each $\alpha$ in 	an $L_\infty$ ball of $0$ in
$\mathbb{R}^{n-1}$ of radius $M$, $P_{\alpha}(\mathcal{G}_n(p)) \leq e^{2M n - c n^2}$, which vanishes provided $M = o(n)$.
\end{proof}

The case of $p=1/2$ was not covered by the previous result and one might wonder
whether setting all the natural parameters to be equal to a constant vector  i.e.
$\alpha_i = \theta$ for all $i$, might yield a model close to an ER mode with $p=1/2$. This in
fact not the case, as we demonstrate next, unless $\theta = 0$.

Towards this end, notice first that the assumption that $\alpha_i = \theta$, for all $i \geq 1$, is equivalent to $\log\left( \frac{p_k}{p_0} \right) = \theta$. Thus, for any $\theta \in \mathbb{R}$,  the probability of observing a graph $g$ is $\frac{e^{-\theta n_0(g)}}{\sum_{j=0}^{n} e^{-\theta j} \nu_n(j)}$,
where $\nu_n(j)$ is the number of $n$-graphs with $j$ isolated nodes. It is
clear that the above probability is rather different from the ER model with
$p=1/2$. To add more details, it is
possible to show that $\nu_n(0) \equiv f(n) = \sum_{k=0}^n (-1)^{n-2} {n \choose k} 2^{ {n \choose k} }$, and $\nu_n(j) = {n \choose j} f(n-j)$.
    This can be obtained as the solution to the recursion
$f(n) = 	2^{ {n \choose 2}  } -  \sum_{i=0}^{n-1}f(i) {n \choose i}$, with the initial conditions $f(0) = 1$ and $f(1) = 0$.

The term $\nu_n(0)$ dominates all the others, in the sense that $\frac{f(n)}{2^{ {n \choose 2} }} \rightarrow 1$, as $n \rightarrow \infty$. Thus, when $\theta < 0$ the model will favor networks with many isolated nodes.



\section{The $2$K model}
\paragraph{The exponential family form.} Consider the lexicographically ordered vector of possible degrees $((1,1),(1,2),\dots,(n-1,n-1))$ in $g$. For $n_{k_1k_2}(g)$, the number of edges with endpoint degrees $k_1$ and $k_2$ in $g$, it holds that $\sum_{k_1=1}^{n-1}\sum_{k_2=k_1}^{n-1} n_{k_1k_2}(g)=e(g)$, where $e(g)$ is the number of edges of $g$, which is still a random quantity. We present an observation with regard to $n_{k_1k_2}(g)$ as the following proposition:
\begin{prop}\label{prop:8}
The following identity always holds:
\begin{equation}\label{eq:13}
\sum_{k_1=1}^{n-1}\sum_{k_2=k_1}^{n-1} \frac{k_1+k_2}{k_1k_2}n_{k_1k_2}(g)=n-n_0(g).
\end{equation}
\end{prop}
\begin{proof}
From (\ref{eq:01}) and (\ref{eq:02}), we obtain
$n-n_0(g)=\sum_{k_1=1}^{n-1}\frac{1}{k_1}(\sum_{k_2=1}^{k_1}n_{k_2k_1}(g)+\sum_{k_2=k_1}^{n-1}n_{k_1k_2}(g))=
\sum_{k_1=1}^{n-1}\sum_{k_2=k_1}^{n-1}(\frac{1}{k_1}n_{k_1k_2}(g)+\frac{1}{k_2}n_{k_1k_2}(g))=\sum_{k_1=1}^{n-1}\sum_{k_2=k_1}^{n-1} \frac{k_1+k_2}{k_1k_2}n_{k_1k_2}(g)$.
\end{proof}

In order to write the model in minimal form and since $\sum_{k_1<k_2}n_{k_1k_2}(g)=e(g)$ is still a random quantity, we use the above proposition and let a scaled version of the joint degree distribution, $\tilde{n}_{k_1,k_2}(g)=((k_1+k_2)/(k_1k_2))n_{k_1k_2}(g)$, be the sufficient statistics.  Therefore, assuming that the probability of observing graphs with isolated nodes is zero, the probability of observing $g$ can be written as
\begin{equation}\label{eq:21}
P(g)=\varphi(p)\prod_{k_1=1}^{n-1}\prod_{k_2=k_1}^{n-1}p_{k_1k_2}^{\tilde{n}_{k_1k_2}(g)},
\end{equation}
where $\varphi$ is a \emph{normalizing constant} and $\sum_{k_1=1}^{n-1}\sum_{k_2=k_1}^{n-1} p_{k_1k_2}=1$.

To write (\ref{eq:21}) in a minimal form, we use the parametrization $\tilde{p}_{k_1k_2}=p_{k_1k_2}/p_{n-1,n-1}$, which by using the fact that $\tilde{p}_{n-1,n-1}=1$, implies $p_{n-1,n-1}=1/(1+\sum_{k_1=1}^{n-2}\sum_{k_2=k_1}^{n-1}\tilde{p}_{k_1k_2})$. Therefore, (\ref{eq:21}) can be rewritten as
\begin{equation}\label{eq:22}
P(g)=\varphi(p)p_{n-1,n-1}^{\sum_{k_1=1}^{n-1}\sum_{k_2=k_1}^{n-1}\tilde{n}_{k_1k_2}(g)}\prod_{k_1=1}^{n-1}\prod_{k_2=k_1}^{n-1}\tilde{p}_{k_1k_2}^{\tilde{n}_{k_1k_2}(g)}=
\frac{\phi(\tilde{p})}{(1+\sum_{k_1=1}^{n-2}\sum_{k_2=k_1}^{n-1}\tilde{p}_{k_1k_2})^n}\prod_{k_1=1}^{n-2}\prod_{k_2=k_1}^{n-1}\tilde{p}_{k_1k_2}^{n_{k_1k_2}(g)}.
\end{equation}
This can be written in exponential family form
\begin{equation}
    P(g)=\exp\Big\{\sum_{k_1=1}^{n-2}\sum_{k_2=k_1}^{n-1}\tilde{n}_{k_1k_2}(g)\alpha_{k_1k_2}-\psi(\alpha)\Big\},
\end{equation}
where $\alpha_{k_1k_2}=\log \tilde{p}_{k_1k_2}$ and $\psi(\alpha)=n\log(1+\sum_{k_1=1}^{n-2}\sum_{k_2=k_1}^{n-1}\tilde{p}_{k_1k_2})-\log\phi(\tilde{p})$.

The sufficient statistics are in fact $\tilde{n}^{(2)}_-(g)$, referring to the vector $\tilde{n}^{(2)}(g)$ with the arbitrary element $\tilde{n}_{n-1,n-1}$ removed.
\paragraph{Calculating the normalizing constant.} Similar to the $1$K case, for a fixed $n$, $\psi(\alpha)=\psi_n(\alpha)$ can be explicitly represented.
In this case the normalizing constant can be calculated directly from the set of graphs with $n$ nodes without isolated nodes, which are denoted by $\mathcal{G}_{n_-}=\{g_1,\dots,g_L\}$.
\begin{prop}\label{prop:21}
The normalizing constant for the $2$K model can be written as
\begin{equation}\label{eq:25}
\psi(\alpha)=\log\Big(\sum_{l=1}^L\prod_{k_1=1}^{n-2}\prod_{k_2=k_1}^{n-1}\tilde{p}_{k_1k_2}^{\tilde{n}_{k_1k_2}(g_l)}\Big)=\log\Big(\sum_{l=1}^Le^{\sum_{k_1=1}^{n-2}\sum_{k_2=k_1}^{n-1}\tilde{n}_{k_1k_2}(g_l)\alpha_{k_1k_2}}\Big).
\end{equation}
\end{prop}
\begin{proof}
The proof is similar to the proof of Proposition \ref{prop:1} by using (\ref{eq:22}).
\end{proof}
\begin{example}\label{ex:21}
From the graphs in Figure 2, it is apparent that there are $2$ non-isomorphic graphs with $3$ nodes and without isolated nodes($g_1$ and $g_2$, where $g_2$ is repeated $3$ times). Proposition \ref{prop:21} implies that
$\psi(\alpha)=\log(1+3\tilde{p}_{12}^2)=\log(1+3e^{2\alpha_{12}})$.
\end{example}

\paragraph{Discussion on the existence of the maximum likelihood estimator and parameter estimation.} Suppose that $\{g_1,\dots,g_m\}$, $m\geq 1$ are $m$ iid observations of the networks with $n$ nodes and without isolated nodes. The \emph{log-likelihood} function can be written as
$l({\alpha|g_1,\dots,g_m})=(\sum_{k_1=1}^{n-2}\sum_{k_2=k_1}^{n-1}\alpha_{k_1k_2}\sum_{l=1}^m\tilde{n}_{k_1k_2}(g_l))-m\psi(\alpha)$.

Similar to the $1$K model, we study the model polytope $A_{n-1,n-2}=convhull(\{\tilde{n}^{(2)}_-(g),g\in\mathcal{G}_n\})$ to derive a necessary and sufficient condition for the existence of the MLE. Notice that the $A_{n-1,n-2}$ lies on an $n(n-1)/2-1$-dimensional space. Finding this polytope is subject to further work.

In general, however, it is easy to observe that if there are no observations of specific bi-degrees, i.e.\ $n_{k_1k_2}=0$, the corresponding parameter $\alpha_{k_1k_2}$ is not estimable. For a single observation, we conjecture that asymptotically, there are at most half of the parameters estimable: It is easy to verify that the graph $T_n$ defined in Proposition \ref{prop:4} has $(n-1/2)^2$ for $n$ odd, and $(n-2/2)^2+n/2$ for $n$ even, non-zero elements in its bi-degree vector. The ratio of the non-zero elements to the length of the vector $n(n+1)/2$ tends to $1/2$ when $n$ tends to infinity, which gives a lower bound for the ratio.

On the other hand, we conduct the following numerical experiment: By using Proposition \ref{prop:8}, we observe that in the vector with maximum non-zeros, ideally every element is $1$. Hence we compute all $(k_1+k_2)/(k_1k_2)$, order them and start adding them until we reach $n$. The number of elements we have added gives an upper bound for the number of non-zero elements of the bi-degree vector. Figure 4 illustrates the ratio to the vector length,
which is decreasing below $0.56$.
\begin{figure}[H]\label{fig:g}
\centering
\includegraphics[scale=0.3]{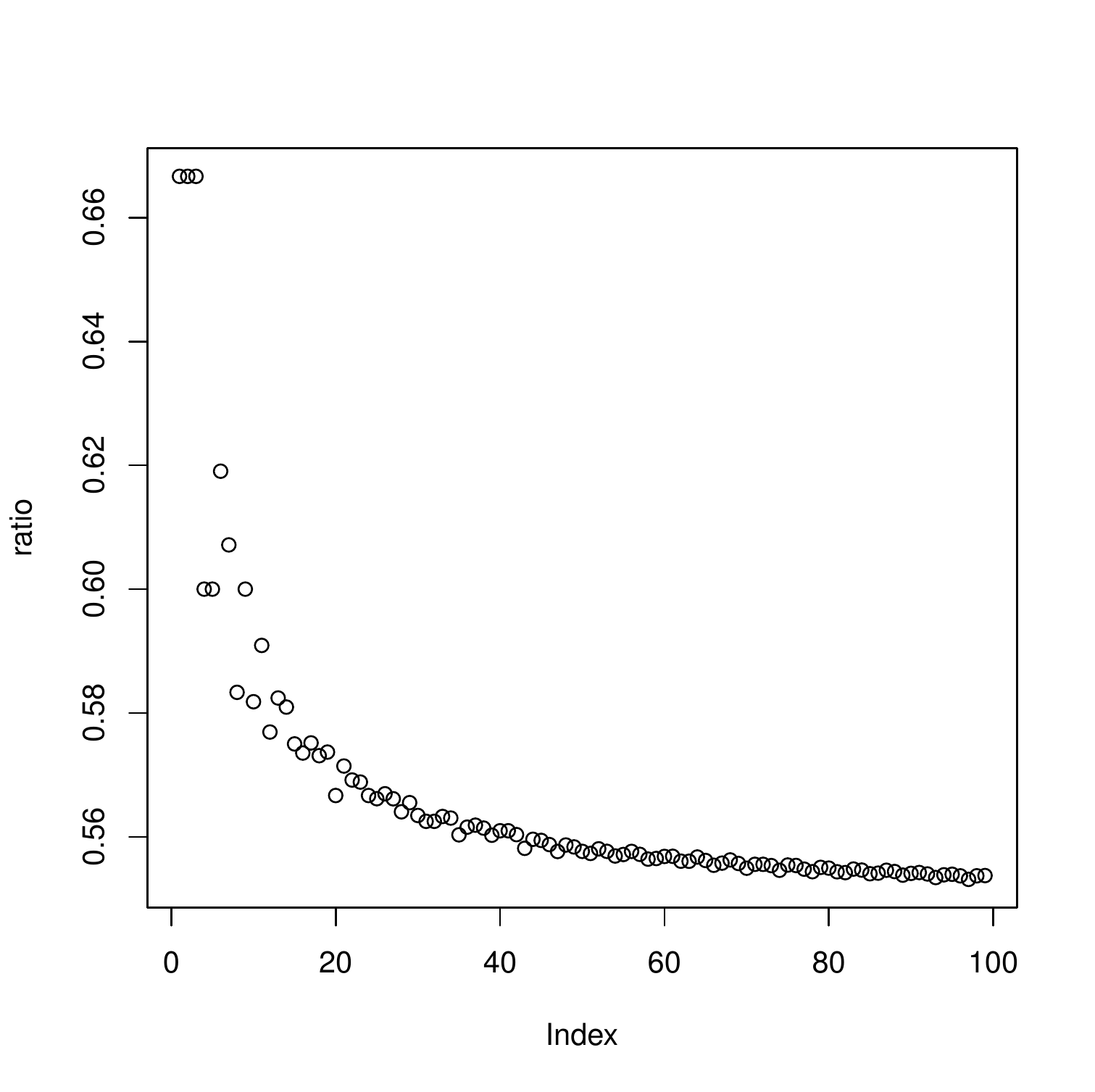}
\caption{{\footnotesize An upper bound for the ratio of maximum non-zero elements of the bi-degree vector to its length.}}
\end{figure}

\bibliographystyle{plain}
\bibliography{bib}

\end{document}